\documentclass{article}

\usepackage{graphicx}
\usepackage{amsmath,amsfonts,amssymb,amsbsy,amsthm}
\usepackage{algorithm,algorithmic}

\def \RR {{\mathbb{R}}}
\def \EE {{\mathbb{E}}}
\def \VV {{\mathbb{V}}}

\def \D  {{\mbox{d}}}

\def \eps {{\varepsilon }}

\def \Ito {{It\^{o} }}

\newcommand{\fracs}[2]{{\textstyle \frac{#1}{#2}}}

\def \hu {{\widehat{u}}}
\def \hE {{\widehat{E}}}
\def \hX {{\widehat{X}}}

\def \htau {{\widehat{\tau}}}
\def \hP {{\widehat{P}}}

\def \of {{\overline{f}}}
\def \oP {{\overline{P}}}
\def \otau {{\overline{\tau}}}
\def \utau {{\underline{\tau}}}

\def \exit {{\rm exit}}

\newtheorem{theorem}{Theorem}[section]
\newtheorem{lemma}[theorem]{Lemma}
\newtheorem{corollary}[theorem]{Corollary}

\begin{document}

\title{Multilevel estimation of expected exit times
       and other functionals of stopped diffusions}
\author{M.B.~Giles \and F.~Bernal}

\maketitle

\begin{abstract}
This paper proposes and analyses a new multilevel Monte Carlo method 
for the estimation of mean exit times for multi-dimensional Brownian 
diffusions, and associated functionals which correspond to solutions 
to high-dimensional parabolic PDEs through the Feynman-Kac formula.
In particular, it is proved that the complexity to achieve an $\eps$
root-mean-square error is $O(\eps^{-2}\, |\!\log \eps|^3)$.

(To appear in SIAM/ASA Journal on Uncertainty Quantification)
\end{abstract}

\section{Introduction}

In this paper we are concerned with solutions of a
stochastic differential equation (SDE)
\begin{equation}
\D X_t = a(X_t, t) \, \D t + b(X_t, t)\, \D W_t,
~~~~~~ 0<t\leq T,
\end{equation}
with deterministic initial condition $X_0 \!=\! x_0$.
We assume that $X_t \!\in\! \RR^d$ and $W \!=\! \{W_t: t\!\geq\! 0\}$ is 
a standard $d'$-dimensional Brownian motion.
Furthermore, we assume that 
$a:\RR^d \times [0,T] \rightarrow \RR^d$ and 
$b:\RR^d \times [0,T] \rightarrow \RR^d\times \RR^{d'}$
are Lipschitz continuous with Lipschitz constants $L_a, L_b$ such that
\begin{eqnarray*}
\| a(x,t) - a(y,s) \|_2 &\leq& L_a \left( \|x\!-y\|_2 + |t\!-\!s| \right),\\
\| b(x,t) - b(y,s) \|_2 &\leq& L_b \left( \|x\!-y\|_2 + |t\!-\!s| \right),~~~
\forall (x,t),(y,s) \in \RR^d \times [0,T],
\end{eqnarray*}
and also that $b\, b^T$ is everywhere non-singular.

Given a simply-connected bounded open domain 
$D\subset \RR^d$, and a fixed point $x_0 \!\in\! D$,
we are interested in estimating the expected exit time 
$\EE^{x_0,0}[\tau]$, 
where the suffix on $\EE^{x,t}$ for time $t: 0\!\leq\!t\!\leq\!T$ 
indicates it is an expectation conditional on the path $X_s, s\!>\!t$, 
starting at $X_t \!=\! x \in D$, 
and the corresponding exit time $\tau$ is defined as 
$\tau \!=\! \min( T, \inf \{ s\!>\!t: X_s \!\notin\! D \} )$.

In addition, we are interested in estimating $u(x_0,0)$ which 
is defined by
\begin{equation}
u(x,t) = \EE^{x,t}\left[ \int_t^\tau E(t,s)\, f(X_s,s) \, \D s 
+ E(t,\tau)\, g(X_\tau,\tau) \right],
\label{eq:functional}
\end{equation}
where
\[
E(t_0,t_1) = \exp\left( - \int_{t_0}^{t_1} V(X_t,t) \, \D t \right).
\]
We assume that the functions
$f:\RR^d \times [0,T] \rightarrow \RR$, 
$g:\RR^d \times [0,T] \rightarrow \RR$, 
$V:\RR^d \times [0,T] \rightarrow \RR$, are all Lipschitz 
continuous with Lipschitz constants $L_f, L_g, L_V$ respectively, 
and also that $g\!\in\! C^{2,1}(D\times[0,T])$, 
with its spatial Hessian $H_g$ and time derivative 
$\dot{g}\!\equiv\! \partial g/\partial t$ both bounded.

The Feynman-Kac formula \cite{kac49,ks91} establishes that 
if the boundary $\partial D$ is sufficiently smooth, $u(x,t)$ 
satisfies the PDE
\begin{equation}
\frac{\partial u}{\partial t} + 
\sum_j a_j(x,t)\frac{\partial u}{\partial x_j} + 
\fracs{1}{2}\! \sum_{j,k,l} b_{j,k}(x,t) b_{l,k}(x,t) 
\frac{\partial^2 u}{\partial x_j \partial x_l}
- V(x,t)\, u(x,t) + f(x,t) = 0,
\label{eq:feynman-kac}
\end{equation}
inside the domain $D \times [0,T)$, subject to Dirichlet boundary 
conditions $u(x,t) \!=\! g(x,t)$ 
when either $x\!\in\! D, t\!=\!T$ or  $x\!\in\! \partial D, 0\!<\!t\!<\!T$.

When using an Euler-Maruyama discretisation for the SDE, with a uniform 
timestep of size $h$, standard analysis gives an $O(h^{1/2})$ strong error
with regards to the path approximation,
and Gobet \& Menozzi \cite{gm10} 
have proved that the weak error for the expected stopping time
is also $O(h^{1/2})$.  To achieve a root-mean-square error of $\eps$
requires $O(\eps^{-2})$ path samples, with $h\!=\!O(\eps^2)$ so that the 
average cost of each path sample is $O(\eps^{-2})$.  Hence the total
computational cost is $O(\eps^{-4})$.

The $O(h^{1/2})$ weak error is due to the $O(h^{1/2})$ movement in the 
SDE solution within each timestep; this is not accounted for by the 
standard piecewise constant interpolation of the Euler-Maruyama
discretisation.  One possible improvement would be the use of a Brownian
Bridge interpolant \cite{glasserman04} which would allow one to sample 
the minimum distance to smooth boundaries within each timestep, 
improving the weak order to $O(h)$.  Another approach, developed by 
Gobet \& Menozzi \cite{gm10}, based on an original idea by Glasserman, 
Broadie \& Kou \cite{bgk97}, introduces a boundary correction, offsetting 
the boundary in the normal direction by a distance which is $O(h^{1/2})$.
This too improves the weak error to $O(h)$, and so the overall 
computational cost is reduced to $O(\eps^{-3})$.

Building on Giles' multilevel Monte Carlo (MLMC) method for barrier 
options \cite{giles08b}, Primozic \cite{primozic11} developed a 
MLMC algorithm for estimating the mean exit time for 1D diffusions 
with planar boundaries.  This used the Milstein discretisation which 
gives an $O(h)$ strong error, combined with a Brownian Bridge estimate 
of the probability of crossing the boundary within each timestep.  
This gives an $O(\eps^{-2})$ computational complexity, but its 
generalisation to multidimensional applications is limited to those 
cases in which the underlying SDE satisfies a commutativity condition 
which means that the L{\'e}vy areas for each timestep are not required 
in the implementation of the Milstein discretisation \cite{glasserman04}.

Higham {\it et al} \cite{hmrsy13} instead proposed a MLMC method 
based on the Euler-Maruyama discretisation, and proved that its 
complexity is $O(\eps^{-3}\, |\!\log \eps|^3)$.  If they had used
either of the two methods mentioned previously to make the 
weak error first order, this would improve to 
$O(\eps^{-2.5}\, |\!\log \eps|^3)$.

The objective in this paper is to improve the numerical algorithm
analysed by Higham {\it et al} to reduce the computational complexity
to $O(\eps^{-2}\, |\!\log \eps|^3)$.  This is achieved by addressing 
the problem identified in the analysis in \cite{hmrsy13}, which is 
the poor decay in the variance of the multilevel correction as 
$h\!\rightarrow\! 0$.  This is approximately $O(h^{1/2})$ because there 
is an $O(h^{1/2})$ probability that a path which starts within 
$O(h^{1/2})$ of the boundary will continue for a time which is $O(1)$ 
before exiting the domain.  

We fix this problem by splitting the fine or coarse path simulation 
into multiple independent copies once the other one has exited the 
domain.  Averaging over these sub-paths approximates the conditional 
expectation, and reduces the multilevel variance to approximately $O(h)$.
However, the expected value on each level is not changed, and 
therefore the telescoping summation which lies at the hear of MLMC 
remains valid.

The paper begins with a lemma bounding the variance of the Feynman-Kac
functional, followed by a review of the Euler-Maruyama approximation of 
the SDE path and the path functional, the Gobet \& Menozzi technique for 
improving the weak convergence, and the MLMC algorithm used by 
Higham {\it et al}.

The new algorithm is then presented and analysed.  The numerical 
analysis relies heavily on the theoretical results of Gobet \& 
Menozzi \cite{gm10}, 
and is similar in structure to the analysis 
of Higham {\it et al} \cite{hmrsy13}.  The effectiveness of the 
new algorithm is demonstrated through numerical experiments, and 
directions for future research are discussed in the final conclusions.

\section{Exit times and Feynman-Kac functionals}

If we define $u_{\exit}(x,t)$ to correspond to the 
particular Feynman-Kac functional in which
$f_{\exit}(x,t)\equiv 0$, $g_{\exit}(x,t)\equiv t$, $V_{\exit}(x,t)\equiv 0$, 
then $u_{\exit}(x_0,0)$ is equal to 
the expected exit time $\EE^{x_0,0}[\tau]$ defined previously. 

We now make a key assumption.

{\bf Assumption 1}: There are Lipschitz constants $L_u$, $L_{\exit}$ 
such that
\begin{eqnarray*}
| u(x,t) - u(y,s) | &\leq& L_u \left( \|x\!-y\|_2 + |t\!-\!s| \right),\\
| u_{\exit}(x,t) - u_{\exit}(y,s) | &\leq& L_{\exit} \left( \|x\!-y\|_2 + |t\!-\!s| \right),~~~
\forall (x,t),(y,s) \in D \times [0,T].
\end{eqnarray*}

\vspace{0.1in}

Comment: If the boundary $\partial D$ is sufficiently smooth, the 
Lipschitz conditions for $u(x,t)$ and $u_{\exit}(x,t)$ in Assumption 1 
follow from the assumed Lipschitz properties for $a, b, f, g, V$.  
However, this assumption might not be satisfied if there is a re-entrant 
corner within the domain.

The following lemma bounds the variance of the SDE functional
\[
P_{t_0} = \int_{t_0}^\tau E(t_0,s)\, f(X_s,s) \, \D s 
+ E(t_0,\tau)\, g(X_\tau,\tau).
\]
The equivalent theorem for the Euler-Maruyama discretisation will
play a critical role in the later numerical analysis.

\begin{lemma}
\label{lemma:var}
There exists a constant $c$ such that for any 
$x_0\in D, 0\!\leq\! \!t_0\!<\!T$
\[
\VV^{x_0,t_0}[ P_{t_0} ] \leq 
c \ \EE^{x_0,t_0}[ \tau \!-\! t_0 ].
\]
\end{lemma}
Note: $\VV^{x_0,t_0}$ represents the variance conditional on $X_{t_0}\!=\!x_0$.
\begin{proof}
By \Ito calculus,
\begin{eqnarray*}
\lefteqn{ \D \left( \rule{0in}{0.16in} E(t_0,s)\, g(X_s,s) \right) = } && \\
&& E(t_0,s) \left( \rule{0in}{0.16in} 
 \left(-\, V\, g + \dot{g} + (\nabla g)^T\, a
+\, \fracs{1}{2}\, \mbox{trace}(b^T H_g\, b)\right) \, \D s
+  (\nabla g)^T\, b \, \D W_s \right),
\end{eqnarray*}
with $a$, $b$, $g$, $\dot{g}\equiv \partial g/\partial t$, $\nabla g$, 
the spatial gradient of $g$, and $H_g$, the Hessian of $g$, 
all evaluated at $(X_s,s)$.  Hence,
\[
P_{t_0} - g(x_0,t_0) = p^{(1)} + p^{(2)},
\]
where
\begin{eqnarray*}
p^{(1)} &=& \int_{t_0}^\tau E(t_0,s) \left(
f - V\, g + \dot{g} + (\nabla g)^T a
+\, \fracs{1}{2}\, \mbox{trace}(b^T H_g\, b) 
\right) \D s,
\\
p^{(2)} &=&  \int_{t_0}^\tau E(t_0,s) \, (\nabla g)^T\, b \ \D W_s.
\end{eqnarray*}

We then have
\[
\VV^{x_0,t_0}[ P_{t_0} ]
\ \leq\  
\EE^{x_0,t_0}[ (P_{t_0}\!-\!g(x_0,t_0))^2 ]
\ \leq\  
2\, \EE^{x_0,t_0}[ (p^{(1)})^2 +  (p^{(2)})^2 ].
\]
Since $E(t_0,s)\leq \exp(T \|V\|_\infty)$, and every other term
in the integrand for $p^{(1)}$ is similarly bounded, there exists 
a constant $c$, independent of $W, x_0, t_0$, such that
$
\left| p^{(1)} \right| \leq c \ (\tau\!-\!t_0)  \leq c \ T
$
and hence
\[
\EE^{x_0,t_0} [ (p^{(1)})^2 ] \leq  c^2\, T\ 
\EE^{x_0,t_0} [ \tau\!-\!t_0  ].
\]
In addition,
\begin{eqnarray*}
\EE^{x_0,t_0}[ (p^{(2)})^2 ] &=& 
\EE^{x_0,t_0}\left[ \int_{t_0}^\tau (E(t_0,s))^2\ \| (\nabla g)^T\, b \|_2^2 \ \D s
\right] \\
& \leq & 
\exp(2 T \|V\|_\infty)\ \| \nabla g \|^2_{2,\infty} \| b \|^2_{2,\infty}
\ \EE^{x_0,t_0}[ \tau\!-\!t_0  ],
\end{eqnarray*}
where $\| b \|_{2,\infty}$,  $\| \nabla g \|_{2,\infty}$ are the 
maximum values of $\| b \|_2$, $\| \nabla g \|_2$ over $D\times [0,T]$. 

This completes the proof.
\end{proof}

\section{Euler-Maruyama approximation}

Using a uniform timestep of size $h$, the standard Euler-Maruyama
approximation of the SDE is
\[
\hX_{t_{n+1}} = \hX_{t_n} + a(\hX_{t_n}, t_n) \, h +  b(\hX_{t_n}, t_n) \, \Delta W_n,
\]
where $\Delta W_n \!\equiv\! W_{t_{n+1}}\!-\!W_{t_n}$ and 
each component of $\Delta W_n$ is a $N(0,h)$ i.i.d.~random variable.
Let $\hX_t$ be the piecewise-constant interpolation in which $\hX_t$ is taken 
to be constant on the open-ended time interval $[t_n, t_{n+1})$.
Correspondingly, for any time interval $[t_0, t]$ with $t_0\!=\!nh$ 
for some integer $n$, we can define
\[
\hE(t_0,t) = \exp\left( - \int_{t_0}^{t} V(\hX_s,s) \, \D s \right),
\]
and let $\hu(x_0,t_0)= \EE^{x_0,t_0}[ \hP_{t_0} ]$ 
where now the suffix on $\EE^{x_0,t_0}$ indicates that the expectation 
is conditional on $\hX_{t_0}\!=\!x_0$,
and $\hP_{t_0}$ is defined by
\[
\hP_{t_0} = \int_{t_0}^{\htau} \hE(t_0,s)\, f(\hX_s,s) \, \D s
\ +\ \hE(t_0,\htau)\, g(\hX_{\htau},\htau),
\]
with $\htau = \min(T, \min \{t_n : \hX_{t_n} \!\notin\! D \})$
being the exit time of the numerical approximation.

$\hu_{\exit}(x_0,t_0)$ can be defined analogously as the expected exit 
time for the numerical path approximation starting from $\hX_{t_0}\!=\!x_0$.

This Euler-Maruyama discretisation is the basis for the MLMC method 
proposed in this paper.  In discussing the computational complexity of
the algorithm, we make the following assumption:

\vspace{0.1in}

{\bf Assumption 2}: There is a unit computational cost in performing
one timestep of the Euler-Maruyama discretisation, and in determining 
whether or not $\hX_{t_{n+1}} \!\in\! D$.

\vspace{0.1in}

Comment: the main point in this assumption is that the cost does not 
depend on the timestep $h$, or the overall accuracy $\eps$ to be achieved.
Regarding the dependence of the cost on the dimension $d$, the actual 
computing time required to generate
the Brownian increments $\Delta W_n$ for one timestep will be proportional 
to $d$, while the computation of $b(\hX_{t_n}, t_n)$ and the product
$b(\hX_{t_n}, t_n)\, \Delta W_n$ will have a cost proportional to $d\, d'$
if $b(x,t)$ is dense.

Determining whether or not $\hX_{t_{n+1}} \!\in\! D$ may not be easy in practice.
It obviously depends on the way in which $D$, or its boundary $\partial D$, 
is specified.  If $\partial D$ is specified as a collection of patches, 
then a generalised octree data structure may be required to minimise the 
searching required to determine which boundary patches a point is near to,
and whether the move from $\hX_{t_n}$ to $\hX_{t_{n+1}}$ has crossed any 
patches.  This cost will actually reduce slightly as $h\!\rightarrow\! 0$ 
since the jumps $\hX_{t_{n+1}} \!-\! \hX_{t_n}$ will become smaller and so less
searching is likely to be required.


Standard theory on the numerical analysis of the Euler-Maruyama 
discretisation \cite{muller02} gives the following strong
convergence result.

\begin{lemma}
\label{lemma:strong}
For $q\!\geq\! 1$ we have
\[
\EE^{x_0,0}\left[ \sup_{[0,T]} \| X_t \!-\! \hX_t \|^q \right]^{1/q}
= \ O( h^{1/2} |\log h|^{1/2}).
\]
\end{lemma}

The following corollary bounds the step changes in the Euler-Maruyama 
solution.

\begin{corollary}
\label{cor:strong}
For $q\!\geq\! 1$ we have 
\[
\EE^{x_0,0}\left[ \sup_{[0,T]} \| \hX_t \!-\! \lim_{s\rightarrow t-} \hX_s\|^q \right]^{1/q}
=\ O( h^{1/2} |\log h|^{1/2}).
\]
\end{corollary}
\begin{proof}
$
\| \hX_t - \hX_s \| \leq \| X_t - \hX_t \| + \| X_s - \hX_s \| + \| X_t - X_s \|,
$
and then the result follows by taking the limit $s\rightarrow t-$ and using
the result from the previous lemma.
\end{proof}

The final assumption concerns the weak error due to the Euler-Maruyama
approximation:

\vspace{0.1in}

{\bf Assumption 3}:
There exist constants $c_u$, and $c_{\exit}$ such that
for all $x\!\in\! D, 0\!<\!t\!<\!T$
\begin{eqnarray*}
\left| u(x,t) \!-\! \hu(x,t) \right| & \leq & c_u\ h^{1/2} \\[0.05in]
\left| u_{\exit}(x,t) \!-\! \hu_{\exit}(x,t) \right| & \leq & c_{\exit}\ h^{1/2}
\end{eqnarray*}

\vspace{0.1in}

Comment: Gobet and Menozzi \cite{gm10} have proved this is true given
Assumption 1 and conditions on the smoothness of the boundary $\partial D$.
See also the related matched asymptotic analysis by Howison and Steinberg 
\cite{hs07}, and a new paper by Bouchard, Geiss and Gobet \cite{bgg17}
which analyses the expected $L_1$ error in the exit time, under 
certain conditions.

Under these conditions, based on an original idea due to Broadie, 
Glasserman \& Kou \cite{bgk97}, Gobet and Menozzi go on to develop a 
numerical discretisation with an improved first order weak convergence, 
by defining the path to have exited the domain at time $t_n$ if 
\underline{either} $\hX_{t_n} \!\notin\!D$, 
\underline{or} $\hX_{t_n} \!\in\!D$ and 
\[
\| \hX_{t_n} \!-\! \pi(\hX_{t_n}) \|_2 
\ \leq \ 
c_0 \, \| n^T(\pi(\hX_{t_n}))\, b(\hX_{t_n}) \|_2 \, h^{1/2}
\]
where $\pi(x) \equiv \arg\min_{y\in\partial D}\| x \!-\! y \|_2$ 
is the projection of $x$ onto the boundary $\partial D$, 
$n(x)$ is the unit normal on the boundary, 
and $c_0 \!=\! -\zeta(1/2)/\sqrt{2 \pi} \!\approx\! 0.5826$
(see (2.1) in \cite{gm10}).


The Lipschitz assumption for $u_{\exit}$, together with the 
bound in Assumption 3, leads immediately to the following 
bound on the expected exit time of numerical path simulations,
based on the initial distance from the boundary:
\begin{lemma}
\label{lemma:exit}
Under Assumptions 1 and 3, for any $x_0\!\in\! D$, 
$t_0\!=\!nh < T$ for some integer $n$, and $y\!\in\! \partial D$,
\[
\EE^{x_0,t_0}\left[\, \htau \!-\! t_0 \, \right] \leq 
L_{\exit} \left( \| y \!-\! x_0 \|_2 + | s \!-\! t_0 | \right)
\ +\ c_{\exit}\, h^{1/2}
\]
\end{lemma}

This lemma is important because of the following theorem:

\begin{theorem}
\label{thm:var}
There exists a constant $c$ such that for any $x_0\in D$
with $t_0\!=\!nh < T$ for some integer $n$,
\[
\VV^{x_0,t_0}[ \hP_{t_0} ] \leq 
c \ \EE^{x_0,t_0}[\, \htau \!-\! t_0 ].
\]
\end{theorem}

\begin{proof}
The proof is similar to that for Lemma \ref{lemma:var}.
If we define ${\bf 1}_n$ to be the indicator function 
for $t_n \!<\! \htau$, then 
by considering $\D(\hE(t_0,s)\, g(\hX_s,s))$ we obtain
\begin{eqnarray*}
\hE(t_0,\htau)\, g(\hX_\htau,\htau) - g(x_0,t_0) &=& 
\int_{t_0}^\htau \hE(t_0,s) \left(-V(\hX_s,s)\, g(\hX_s,s) + \dot{g}(\hX_s,s) \right)\ \D s
\\ &+& \sum_{n\geq 0} {\bf 1}_n \hE(t_0,t_{n+1}) 
\left( g(\hX_{t_{n+1}},t_{n+1}) - g(\hX_{t_n}, t_{n+1}) \right).
\end{eqnarray*}

If the function $F: \RR \rightarrow \RR$ is twice continuously
differentiable, then there exists some $0\!<\!\xi\!<\!1$ such that
$F(1) \!-\! F(0) \!=\! F'(0) \!+\! \fracs{1}{2} F''(\xi)$.  Applying
this to $F(s) \equiv g(\hX_{t_n} \!+\! s\,\Delta \hX, t_{n+1})$ 
with
$\Delta \hX \equiv \hX_{t_{n+1}} \!- \hX_{t_n} =
a(\hX_{t_n}, t_n)\, h + b(\hX_{t_n}, t_n) \, \Delta W_n$
gives
\begin{eqnarray*}
 g(\hX_{t_{n+1}},t_{n+1}) - g(\hX_{t_n}, t_{n+1}) &=&
(\nabla g)_n^T ( a_n\, h + b_n \, \Delta W_n ) \\
&+& \fracs{1}{2}  ( a_n\, h + b_n \, \Delta W_n )^T H_{g,n}\, ( a_n\, h + b_n \, \Delta W_n ),
\end{eqnarray*}
where 
$a_n \!\equiv\! a(\hX_{t_n}, t_n)$,
$b_n \!\equiv\! b(\hX_{t_n}, t_n)$, 
$(\nabla g)_n \equiv \nabla g(\hX_{t_n}, t_{n+1})$, 
and $H_{g,n}$, the Hessian of $g$, 
is evaluated at $(\hX_{t_n}\!+\!\xi \Delta \hX, t_{n+1})$.

Hence,
\[
\hP_{t_0} - g(x_0,t_0) = p^{(1)} +  p^{(2)} +  p^{(3)},
\]
where
\begin{eqnarray*}
p^{(1)} &=& 
\int_{t_0}^\htau \hE(t_0,s)\, 
\left( f(\hX_s,s) - V(\hX_s,s)\, g(\hX_s,s) + \dot{g}(\hX_s,s) \right) \, \D s, \\
p^{(2)} &=& 
\sum_{n\geq 0} {\bf 1}_n \hE(t_0,t_{n+1})\, (\nabla g)^T_n b_n \, \Delta W_n, \\
p^{(3)} &=& 
\sum_{n\geq 0} {\bf 1}_n \hE(t_0,t_{n+1})
\left(  (\nabla g)^T_n a_n h 
+ \fracs{1}{2}  ( a_n h \!+\! b_n \Delta W_n )^T H_{g,n}\, 
( a_n h \!+\! b_n \Delta W_n ) \right).
\end{eqnarray*}

The variance therefore has the bound
\begin{eqnarray*}
\VV^{x_0,t_0}[ \hP_{t_0} ]
&\leq&
\EE^{x_0,t_0}\left[ (\hP_{t_0}\!-\!g(x_0,t_0))^2 \right] \\
&\leq&  
3\, \EE^{x_0,t_0}[ (p^{(1)})^2 + (p^{(2)})^2  + (p^{(3)})^2 ].
\end{eqnarray*}
Since the integrand for $p^{(1)}$ is bounded, as before, 
there exists a constant $c$,
independent of $W, x_0, t_0$ and $h$, such that
$
| p^{(1)} | \leq c \ (\htau\!-\!t_0)  \leq c \ T
$
and hence
\[
\EE^{x_0,t_0} [ (p^{(1)})^2 ] \leq  c^2\, T\ 
\EE^{x_0,t_0}[ \htau\!-\!t_0 ].
\]

In addition,
\begin{eqnarray*}
\EE^{x_0,t_0} [ (p^{(2)})^2 ] 
&=& \EE^{x_0,t_0}\left[
\sum_{n\geq 0} {\bf 1}_n (\hE(t_0,t_{n+1}))^2\ 
\left( (\nabla g)_n^T\, b_n \, \Delta W_n \right)^2 \right] \\
&=& \EE^{x_0,t_0}\left[
\sum_{n\geq 0} {\bf 1}_n (\hE(t_0,t_{n+1}))^2\ \| (\nabla g)_n^T\, b_n \|_2^2 \ h
\right] \\
& \leq & 
\exp(2 T \|V\|_\infty)\ \| \nabla g \|^2_{2,\infty} \| b \|^2_{2,\infty}
\ \EE^{x_0,t_0}[ \htau\!-\!t_0 ].
\end{eqnarray*}

Finally, we need to bound $\EE^{x_0,t_0}[ (p^{(3)})^2 ]$. 
There are at most $T/h$ 
timesteps before the computed path leaves the domain, and hence
\[
\EE^{x_0,t_0} [ (p^{(3)})^2 ] \leq 
\EE^{x_0,t_0} \left[ \sum_{n\geq 0} {\bf 1}_n h\, S_n \right]
\]
where 
\[
S_n = T h^{-2} (\hE(t_0,t_{n+1}))^2 
\left( (\nabla g)_n^T\, a_n h 
+ \fracs{1}{2}  ( a_n h \!+\! b_n \Delta W_n )^T H_{g,n}\, 
( a_n h \!+\! b_n \Delta W_n ) \right)^2.
\]
Repeatedly using the inequality $(u\!+\!v)^2 \leq 2(u^2\!+\!v^2)$, 
and bounding the various terms, gives 
${\bf 1}_n S_n \leq {\bf 1}_n S'_n$ where
\begin{eqnarray*}
S'_n &=& 2 \, T\, \exp(2 T \|V\|_\infty)
\left(\ \|\nabla g \|^2_{2,\infty}\, \| a\|^2_{2,\infty} 
\right.
\\ ~~~~ &&
\left. \hspace{1.25in}
+\, \|H_g\|^2_{2,\infty}  
(  T^2 \| a\|^4_{2,\infty} +
h^{-2}\, \| b \|^4_{2,\infty} \| \Delta W_n \|_2^4 )
\ \right).
\end{eqnarray*}
The $S'_n$ are i.i.d.~and have a finite expected value
which is independent of $h$.  In addition, for each $n$, 
$S'_n$ is independent of ${\bf 1}_n$. Hence
\[
\EE^{x_0,t_0} [ (p^{(3)})^2 ]\ \leq\ 
\EE^{x_0,t_0} \left[ \sum_{n\geq 0} {\bf 1}_n h\, S'_n \right]
\ =\ \EE[S'_0] \ \EE^{x_0,t_0}[ \htau-t_0 ].
\]
This completes the proof.
\end{proof}


\section{MLMC algorithm}

\label{sec:MLMC}

The multilevel Monte Carlo (MLMC) method is well documented 
\cite{giles08,giles15}.   The following is the 
version of the MLMC theorem given in \cite{giles15}:

\begin{theorem}
\label{thm:MLMC}
Let $P$ denote a random variable, and let $\hP_\ell$ denote the 
corresponding level $\ell$ numerical approximation.

If there exist independent estimators $Y_\ell$
each of which is an average of $N_\ell$ i.i.d.~Monte Carlo samples, 
each with expected cost $C_\ell$ and variance 
$V_\ell$, and positive constants 
$\alpha, \beta, \gamma, c_1, c_2, c_3$ such that 
$\alpha\!\geq\!\fracs{1}{2}\min(\beta,\gamma)$ and
\begin{itemize}
\item[i)] ~
$\displaystyle
\left|\rule{0in}{0.13in} \EE[\hP_\ell \!-\! P] \right|\ \leq\ c_1\, 2^{-\alpha\, \ell}
$
\vspace{0.02in}
\item[ii)] ~
$\displaystyle
\EE[Y_\ell]\ = \left\{ \begin{array}{ll}
\EE[\hP_0],                     &~~ \ell=0 \\[0.05in]
\EE[\hP_\ell \!-\! \hP_{\ell-1}], &~~ \ell>0
\end{array}\right.
$
\vspace{0.02in}
\item[iii)] ~
$\displaystyle
V_\ell\ \leq\ c_2\, 2^{-\beta\, \ell}
$
\vspace{0.02in}
\item[iv)] ~
$\displaystyle
C_\ell\ \leq\ c_3\, 2^{\gamma\, \ell},
$
\end{itemize}
then there exists a positive constant $c_4$ such that for any 
$\eps \!<\! e^{-1}$
there are values $L$ and $N_\ell$ for which the multilevel estimator
\[
Y = \sum_{\ell=0}^L Y_\ell,
\]
has a mean-square-error with bound
\[
MSE \equiv \EE\left[ \left(Y - \EE[P]\right)^2\right] < \eps^2
\]
with an expected computational complexity $C$ with bound
\[
\EE[C] \leq \left\{\begin{array}{ll}
c_4\, \eps^{-2}              ,    & ~~ \beta>\gamma, \\[0.05in]
c_4\, \eps^{-2} (\log \eps)^2,    & ~~ \beta=\gamma, \\[0.05in]
c_4\, \eps^{-2-(\gamma-\beta)/\alpha}, & ~~ \beta<\gamma.
\end{array}\right.
\]
\end{theorem}

In the setting of this paper,
\begin{equation}
P = \int_0^\tau E(0,s)\, f(X_s,s) \, \D s + E(0,\tau)\, g(X_\tau,\tau),
\label{eq:functional2}
\end{equation}
with each path starting from $X_0\!=\!x_0$,
and $\hP_\ell$ represents the corresponding numerical approximation 
on level $\ell$ using a uniform timestep of size $h_\ell$ defined 
by $h_\ell = K^{-\ell} \, h_0$, for some fixed $h_0\!>\!0$ and 
integer $K\!>\!1$.  Based on the observations in \cite{giles08}, 
we will use $K\!=\!4$ in the algorithm description and in the numerical 
experiments presented later, but the numerical analysis will allow
for the possibility of other values.


\begin{figure}

\begin{center}
\includegraphics{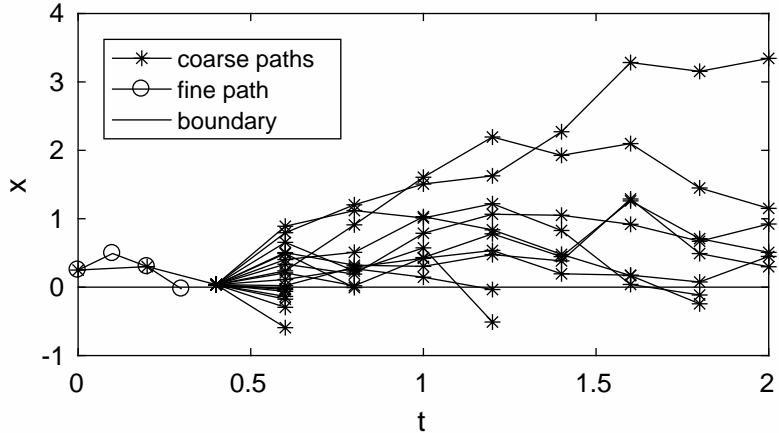}
\end{center}

\caption{Illustration of fine and coarse 1D Brownian paths 
with termination when they cross a boundary at $x\!=\!0$.
Note: the actual algorithm and analysis is based on a 
piecewise constant approximation, but a continuous piecewise 
linear interpolation is used here for visual clarity.}
\label{fig:problem}
\end{figure}

The question now is how to define $Y_\ell$?  Higham {\it et al}
\cite{hmrsy13} define it as the average of the differences in
numerical approximations on levels $\ell$ and $\ell\!-\!1$ 
from $N_\ell$ independent path simulations,
\[
Y_\ell = N_\ell^{-1} \sum_{n=1}^{N_\ell}
\left(\rule{0in}{0.16in} \hP_\ell(W^{(n)}) - \hP_{\ell-1}(W^{(n)})\right).
\]
Here the notation $\hP_\ell(W^{(n)}) - \hP_{\ell-1}(W^{(n)})$ 
indicates the difference of two numerical approximations with 
different timesteps but the same driving Brownian motion.  The 
use of the same Brownian path is key to the multilevel treatment; 
it is this which ensures that the variance $V_\ell$ decays to zero 
as $\ell\!\rightarrow\!\infty$.

However, the variance does not decay to zero as rapidly as 
one would hope.  The problem is illustrated in Figure 
\ref{fig:problem}.  The line with circles represents a fine 
Brownian path which after 3 timesteps crosses the boundary.  
The other line shows that after two coarse timesteps (each 
corresponding to two fine timesteps) the coarse path has
not crossed the boundary, but it is extremely close to the 
boundary.  Continuing the simulation from that point, the
multiple lines correspond to multiple independent simulations
of the extension of the Brownian path.  Many of these quickly
cross the boundary, but some fail to cross until much later.
In fact, speaking rather loosely, for any fixed time $T$ 
there is approximately an $O(h^{1/2})$ probability that the 
coarse path continues for more than time $T$ before crossing 
the boundary.
Higham {\it et al} \cite{hmrsy13} prove that as a consequence,
$V_\ell = O(h_\ell^{1/2} |\log h_\ell|^{1/2})$
and this leads to an MLMC computational complexity bound
which is $O(\eps^{-3} |\log \eps |^3)$.

Higham {\it et al} did not use a boundary correction, and 
therefore their approximation had only $O(h^{1/2})$ weak 
convergence. In the case of smooth boundaries $\partial D$,  
modifying their MLMC treatment to use Gobet \& Menozzi's 
boundary shift \cite{gm10} would improve the weak convergence 
to first order, and reduce the overall complexity to 
$O(\eps^{-2.5} |\log \eps |^3)$ by approximately halving the 
number of levels required to achieve the desired accuracy.

The new MLMC algorithm has the same 
complexity regardless of whether it uses the standard boundary 
treatment or the Gobet \& Menozzi treatment.  The numerical 
analysis will be given in detail for the former, but has a natural 
extension to the latter.  The numerical results for both methods
indicate that the latter is $4-10 \times$ more efficient in practice.

\section{New MLMC algorithm}

\label{sec:new}

Figure \ref{fig:problem} illustrated not only the problem with 
the existing MLMC algorithm, but also the solution.  The key idea
in the new algorithm is that when simulating a pair of coarse and 
fine paths, driven by the same Brownian motion, once one path 
exits the domain the other path is split (or duplicated) into 
multiple copies, each driven from that point onwards by different 
independent Brownian path realisations.  Averaging the output from 
these different copies gives an approximation to the conditional 
expectation for that path, and the variance for the multilevel 
estimator based on the conditional expectation is much lower.
This is similar to the use of conditional expectations for 
digital options and barrier options in \cite{giles08b,giles15}.

By introducing the splitting, it is clear that the computational
cost increases.  This is important in deciding the number of split 
paths.  Ignoring factors proportional to $\log h_\ell$,
it will be shown that the split paths start at a distance
which is $O(h_\ell^{1/2})$ from the boundary, and 
therefore the expected exit time for each of the split paths
is $O(h_\ell^{1/2})$, due to Assumptions 1 and 3.  By making the 
number of split paths $O(h_\ell^{-1/2})$, the cost of the split paths
is $O(h_\ell^{-1})$, which is the same as the order of cost of 
the coarse and fine paths before the split.  Thus, the overall 
cost has been increased by some factor, but its order has not
changed.  It will be shown later, after the proof of Theorem 
\ref{thm:main}, that this number of split 
paths is the minimum required to reduce the multilevel variance 
to the same order as using the true conditional expectation, 
which is equivalent to using an infinite number of split paths.

\begin{algorithm}[t]
\caption{Algorithm to compute $Y_\ell$ using $N_\ell$ pairs of 
coarse/fine paths, with a split into $M_\ell$ sub-paths after one exits}
\label{alg:Yl}

\begin{algorithmic}
\FOR{$n = 1, \ldots, N_\ell$}
\STATE for given $W^{(n)}$, compute $\hX^{(n)}_{\ell}$ and
       $\hX^{(n)}_{\ell-1}$ 
using Euler-Maruyama discretisation with Brownian motion $W^{(n)}$
and timesteps $h_\ell, h_{\ell-1}$ respectively,
       until the end of the first coarse timestep
       for which at least one of the two has exited domain $D$
       or reached the terminal time $T$ \\[0.1in]
\IF{$\hX^{(n)}_{\ell}$ has exited $D$ or the terminal time $T$ has been reached}
\STATE $\oP_\ell^{(n)} := \hP_\ell(W^{(n)})$ based on path already computed
\ELSE
\STATE $\displaystyle \oP_\ell^{(n)} :=  M_\ell^{-1} \sum_{m=1}^{M_\ell} \hP_\ell(W^{(n)},Z^{(n,m)})$
       based on $M_\ell$ independent path\\ ~~~~~~~~ continuations using Brownian paths
       represented by $Z^{(n,m)}$
\ENDIF \\[0.1in]
\IF{$\hX^{(n)}_{\ell-1}$ has exited $D$ or the terminal time $T$ has been reached}
\STATE $\oP_{\ell-1}^{(n)} := \hP_{\ell-1}(W^{(n)})$ based on path already computed
\ELSE
\STATE $\displaystyle \oP_{\ell-1}^{(n)} := M_\ell^{-1} \sum_{m=1}^{M_\ell} \hP_{\ell-1}(W^{(n)},Z^{(n,m)})$
       based on $M_\ell$ independent path\\ ~~~~~~~~ continuations using Brownian paths
       represented by $Z^{(n,m)}$
\ENDIF \\[0.1in]
\ENDFOR \\[0.1in]
\STATE $\displaystyle Y_\ell := N_\ell^{-1} \sum_{n=1}^{N_\ell}  \left(\oP_\ell^{(n)} - \oP_{\ell-1}^{(n)} \right)$
\end{algorithmic}

\end{algorithm}

The algorithm to compute $Y_\ell$ is given in Algorithm \ref{alg:Yl}.
The timestep on level $\ell$ is taken to be $h_\ell \!=\! 4^{-\ell}\, h_0$.
For each of the $N_\ell$ samples, the fine and coarse paths are
computed using the same driving Brownian path $W^{(n)}$ until 
the end of the first coarse timestep for which at least one 
of the two paths has either left the domain $D$ or reached
the terminal time $T$. If the other path has not yet exited $D$
or reached time $T$, it is split into $M_\ell$ copies, and each is 
continued using independent Brownian paths represented by 
$Z^{(n,m)}$ until it finally exits the domain or reaches time $T$.  
The final value for $Y_\ell$ is given by
\[
Y_\ell = N_\ell^{-1} \sum_{n=1}^{N_\ell}
\left( 
\, \oP_\ell(W^{(n)}) \!-\! \oP_{\ell-1}(W^{(n)}) 
\right)
\]
where
$\displaystyle
\oP_\ell(W^{(n)}) = M_\ell^{-1} \sum_{m=1}^{M_\ell} \hP_\ell(W^{(n)},Z^{(n,m)})$,
and $\oP_{\ell-1}(W^{(n)})$ is defined similarly.
With a slight abuse of notation, each of the quantities 
$\hP_\ell(W^{(n)},Z^{(n,m)})$  
is the appropriate approximation of the Feynman-Kac functional 
(\ref{eq:functional2}) as defined in Section \ref{sec:MLMC},
and the averaging over $m$ is trivial for the path which did not 
split and therefore has no dependence on $Z^{(n,m)}$.

\section{Numerical analysis}

\subsection{Variance analysis}

For completeness, the numerical analysis begins with this general 
result for estimators with splitting.  This use of splitting is related 
to the Conditional Monte Carlo method (see pp.145--148 in \cite{ag07}) 
but quite different from the splitting technique in rare event simulation.

\begin{lemma}
\label{lemma:split}
If $W$ and $Z^{(1)}, \ldots ,Z^{(M)}$ are independent 
random variables such that $Z^{(1)}, \ldots,Z^{(M)}$ are identically 
distributed to a random variable $Z$, then
\[
\of = M^{-1} \sum_{m=1}^M f(W,Z^{(m)})
\]
is an unbiased estimator for
$\EE\left[ f(W,Z) \right]$ and its variance is 
\[
\VV[\,\of\,] = \VV\left[\rule{0in}{0.16in} \EE[f(W,Z) \,|\, W] \right]
+ M^{-1}\, \EE\left[\rule{0in}{0.16in} \VV[f(W,Z) \,|\, W] \right].
\]
\end{lemma}

\begin{proof}
Conditional on a value for $W$, the expected value of $\of$ is 
$\EE[f(W,Z) \,|\, W]$ and its variance is $M^{-1}\, \VV[f(W,Z) \,|\, W]$, 
and therefore
\[
\EE\left[\, \of^2 \,|\, W \right] = 
\left(\rule{0in}{0.16in} \EE[f(W,Z) \,|\, W] \right)^2
                 + M^{-1}\, \VV[f(W,Z) \,|\, W].
\]
Taking the expectation over the distribution for $W$ then gives
\begin{eqnarray*}
\EE[\,\of\,] &=& 
\EE \left[\rule{0in}{0.16in} \EE[f(W,Z)\,|\, W]  \right] 
\ = \ \EE[f(W,Z)], 
\\
\EE\left[\, \of^2 \,\right] &=& 
\EE\left[\rule{0in}{0.16in} \left( \EE[f(W,Z)\,|\, W] \right)^2 \right]
        + M^{-1}\, \EE\left[\rule{0in}{0.16in} \VV[f(W,Z)\,|\, W] \right],
\end{eqnarray*}
from which it follows that
\[
\VV[\,\of\,]
\ =\ \EE\left[\, \of^2 \,\right]
  - \left(\rule{0in}{0.16in} \EE[\,\of\,]\right)^2
= \VV\left[\rule{0in}{0.16in} \EE[f(W,Z) \,|\, W] \right]
    + M^{-1}\,\EE\left[\rule{0in}{0.16in} \VV[f(W,Z) \,|\, W] \right].
\]
\end{proof}

Applying this result to the multilevel estimator in Section 
\ref{sec:new} using $M_\ell$ split paths on level $\ell$, gives us
\begin{equation}
V_\ell =
\VV\left[ \oP_\ell \!-\! \oP_{\ell-1} \right]
= \VV\left[ \rule{0in}{0.16in} \, \EE[\hP_\ell \!-\! \hP_{\ell-1} \, | \, W] \, \right]
+ M_\ell^{-1} \, \EE\left[ \rule{0in}{0.16in} \, \VV[\hP_\ell \!-\! \hP_{\ell-1} \, | \, W] \, \right].
\label{eq:Vl}
\end{equation}
The main theorem will bound 
$\VV[\, \EE[\hP_\ell - \hP_{\ell-1} \, | \, W] \, ]$
and
$\EE[\, \VV[\hP_\ell - \hP_{\ell-1} \, | \, W] \, ]$,
and we will then bound $C_\ell$, the expected cost of each sample 
when using the splitting. To prepare for it, 
we begin by defining $\utau$ to be the exit time of the first of 
a pair of coarse/fine paths, and $\otau$ to be $\utau$ rounded 
up to the end of a coarse timestep (i.e.~a multiple of $h_{\ell-1}$),
which is the time at which the splitting occurs if one of the 
two paths remains within the domain.
Note that $\otau\!\neq\!\utau$ only when the fine path exits 
before the coarse path, part-way through a coarse timestep.

The following lemma bounds the difference between the coarse 
and fine paths up to time $\utau$, and also the difference between 
the fine path at time $\utau$ and the coarse path at the
possibly different time $\otau$.

\begin{lemma}
\label{lemma:diff}
Given the definitions above, for $q\!\geq\! 1$ we have
\[
\EE[\, \sup_{[0,\utau]} \| \hX_{\ell,t} \!-\! \hX_{\ell-1,t} \|^q \,]^{1/q}
= O(h_{\ell-1}^{1/2}\, |\log h_{\ell-1}|^{1/2})
\]
and
\[
\EE[\, \| \hX_{\ell,\utau} \!-\! \hX_{\ell-1,\otau} \|^q \,]^{1/q}
= O(h_{\ell-1}^{1/2}\, |\log h_{\ell-1}|^{1/2})
\]
\end{lemma}

\begin{proof}
These follow directly from Lemma \ref{lemma:strong} and 
Corollary \ref{cor:strong}.
\end{proof}

For a particular pair $(W,Z)$, we define $\htau_\ell$ and 
$\htau_{\ell-1}$ to be the exit times of the fine and coarse paths.
The next lemma bounds the expected absolute difference between
these two exit times. 
This is unaffected by the path splitting, so it is sufficient
to consider what happens when there is no splitting.

\begin{lemma}
\label{lemma:split_exit}
Given the definitions above and Assumptions 1 and 3,
\[
\EE\left[\, |\htau_\ell \!-\!  \htau_{\ell-1}| \,\right] =
O(h_{\ell-1}^{1/2}\, |\log h_{\ell-1}|^{1/2}).
\]
\end{lemma}

\begin{proof}
Each $W$ leads to one of three possibilities:
\begin{itemize}
\item $\hX_{\ell,\utau}\!\notin\! D$ and $\hX_{\ell-1,\otau}\!\notin\! D$
\item $\hX_{\ell,\utau}\!\notin\! D$ and $\hX_{\ell-1,\otau}\!\in\! D$
\item $\hX_{\ell-1,\utau}\!\notin\! D$ and $\hX_{\ell,\utau}\!\in\! D$
\end{itemize}

In the first case, $|\htau_\ell \!-\!  \htau_{\ell-1}| = \otau - \utau < h_{\ell-1}$.

In the second case, since $\hX_{\ell-1,\otau}\!\in\! D$ 
but $\hX_{\ell,\utau}\!\notin\! D$,
there is a point $y\in\partial D$ on the line 
between the two.  Hence, noting that $u_{\exit}(y,\otau)=\otau$, we have
\begin{eqnarray*}
\EE[\htau_{\ell-1}\!-\!\otau \,|\, W] &=&
\left(\hu_{\ell-1,\exit}(\hX_{\ell-1,\otau},\otau) - u_{\exit}(\hX_{\ell-1,\otau},\otau) \right)
\\ &+& \left(u_{\exit}(\hX_{\ell-1,\otau},\otau) - u_{\exit}(y,\otau) \right).
\end{eqnarray*}
The first term is bounded by Assumption 3, and the second term by the 
Lipschitz property in Assumption 1, giving
\[
\EE[\htau_{\ell-1}\!-\!\otau \,|\, W] \leq  
c_{\exit} h_{\ell-1}^{1/2} + L_{\exit} \left( \| \hX_{\ell,\utau}-\hX_{\ell-1,\otau} \|_2 
+ h_{\ell-1} \right).
\]
Adding in the fact that $\otau \!-\! \htau_\ell \leq h_{\ell-1}$, this
can be simplified to there being a constant $c$ for which
\[
\EE[\htau_{\ell-1}\!-\!\htau_\ell \,|\, W] \leq  
c \left( h_{\ell-1}^{1/2} + \|\hX_{\ell,\utau}-\hX_{\ell-1,\otau} \|_2 \right).
\]
A similar result holds in the third case, and hence there is a constant 
$c$ for which
\[
\EE[\htau_\ell \!-\! \htau_{\ell-1} \,|\, W ]
\leq c \left( h_{\ell-1}^{1/2} + \|\hX_{\ell,\utau}-\hX_{\ell-1,\otau} \|_2 \right).
\]
Combining these two results, taking an expectation over $W$, and 
using the result in Lemma \ref{lemma:diff} gives the desired 
result that
\[
\EE\left[\, |\htau_\ell \!-\!  \htau_{\ell-1}| \,\right] =
O(h_{\ell-1}^{1/2}\, |\log h_{\ell-1}|^{1/2}).
\]
\end{proof}

A new paper by Bouchard, Geiss and Gobet \cite{bgg17} analyses 
the expected $L_1$ error in the exit time, and proves that under 
certain conditions $\EE[\,|\htau\!-\!\tau|\,] = O(h^{1/2})$. It follows 
from this that we have the slightly stronger result
$\EE\left[\, |\htau_\ell \!-\!  \htau_{\ell-1}| \,\right] = O(h_{\ell-1}^{1/2})$.

\vspace{0.1in}

We are now ready to prove the main theorem of this section.

\begin{theorem}
\label{thm:main}
Under Assumptions 1 and 3, 
\begin{eqnarray*}
\VV\left[ \rule{0in}{0.16in} \, 
\EE[\hP_\ell \!-\! \hP_{\ell-1} \, | \, W] \, \right]
&=& O(h_{\ell-1} |\log h_{\ell-1}|),   \\
\EE\left[ \rule{0in}{0.16in} \, 
\VV[\hP_\ell \!-\! \hP_{\ell-1} \, | \, W] \, \right]
&=& O(h_{\ell-1}^{1/2} |\log h_{\ell-1}|^{1/2}).
\end{eqnarray*}
\end{theorem}

\begin{proof}
For the fine path,
\begin{eqnarray*}
\hP_\ell &=& \int_0^{\htau_\ell} \hE_\ell(0,s)\, f(\hX_{\ell,s},s) \, \D s
 + \hE_\ell(0,\htau_\ell)\, g(\hX_{\ell,\htau_\ell},\htau_\ell) \\[0.05in]
&=& p_\ell^{(1)}(W) + p_\ell^{(2)}(W)\, p_\ell^{(3)}(W,Z),
\end{eqnarray*}
where
\begin{eqnarray*}
p_\ell^{(1)}(W)   &=& \int_0^{\utau} \hE_\ell(0,s)\, f(\hX_{\ell,s},s) \, \D s, \\[0.05in]
p_\ell^{(2)}(W)   &=& \hE_\ell(0,\utau) \\[0.05in]
p_\ell^{(3)}(W,Z) &=& \int_\utau^{\htau_\ell} \hE_\ell(\utau,s)\, f(\hX_{\ell,s},s) \, \D s
~ + ~ \hE_\ell(\utau,\htau_\ell)\, g(\hX_{\ell,\htau_\ell},\htau_\ell),
\end{eqnarray*}
and similarly for the coarse path,
\begin{eqnarray*}
\hP_{\ell-1} &=& \int_0^{\htau_{\ell-1}} \hE_{\ell-1}(0,s)\, f(\hX_{{\ell-1},s},s) \, \D s
 + \hE_{\ell-1}(0,\htau_{\ell-1})\, g(\hX_{{\ell-1},\htau_{\ell-1}},\htau_{\ell-1}) \\[0.05in]
&=& p_{\ell-1}^{(1)}(W) + p_{\ell-1}^{(2)}(W)\, p_{\ell-1}^{(3)}(W,Z),
\end{eqnarray*}
where
\begin{eqnarray*}
p_{\ell-1}^{(1)}(W)   &=& \int_0^{\otau} \hE_{\ell-1}(0,s)\, f(\hX_{{\ell-1},s},s) \, \D s, \\[0.05in]
p_{\ell-1}^{(2)}(W)   &=& \hE_{\ell-1}(0,\otau) \\[0.05in]
p_{\ell-1}^{(3)}(W,Z) &=& \int_\otau^{\htau_{\ell-1}} \hE_{\ell-1}(\otau,s)\, f(\hX_{{\ell-1},s},s) \, \D s
~ + ~ \hE_{\ell-1}(\otau,\htau_{\ell-1})\, g(\hX_{{\ell-1},\htau_{\ell-1}},\htau_{\ell-1}).
\end{eqnarray*}

Noting that 
$
\exp(v) - \exp(w) = (v\!-\!w) \, \exp(\xi),
$
for some $\xi$ in the interval with endpoints $v, w$, and defining
\[
\| \Delta \hX \|_\infty = \sup_{t\in[0,\utau]} \| \hX_{\ell,t} \!-\! \hX_{\ell-1,t} \|_2,
\]
then for $s\in[0,\utau]$ we have
\[
\left| \hE_\ell(0,s) - \hE_{\ell-1}(0,s) \right| \leq
\exp(T \| V \|_\infty) \ T L_V \| \Delta \hX \|_\infty.
\]
We also have
\[
\left| f(\hX_{\ell,s},s) - f(\hX_{\ell-1,s},s) \right| 
\leq L_f \| \Delta \hX \|_\infty,
\]
and hence there exists a constant $c^{(1)}$ (independent of $W$)
such that
\begin{equation}
\left| p_\ell^{(1)} \!-\! p_{\ell-1}^{(1)}\right|
\ \leq\ 
c^{(1)} \| \Delta \hX \|_\infty + h_{\ell-1}\, \exp(T \| V \|_\infty)\, \| f \|_\infty,
\label{eq:p1}
\end{equation}
with the second term on the right being due to the possible 
difference in the upper limits of integration in the definitions 
of $p_\ell^{(1)}, p_{\ell-1}^{(1)}$.

Similarly, there exists a constant $c^{(2)}$ (independent of $W$)
such that
\begin{equation}
\left| p_\ell^{(2)} \!-\! p_{\ell-1}^{(2)}\right|
\ \leq\ 
c^{(2)} \| \Delta \hX \|_\infty + h_{\ell-1} \, \exp(T \| V \|_\infty)\, \| V \|_\infty.
\label{eq:p2}
\end{equation}

Taking an expectation over $Z$, we have
\begin{eqnarray*}
\EE[ p_\ell^{(3)} \,|\, W ] &=& \hu_\ell(\hX_{\ell,\utau},\utau), \\ 
\EE[ p_{\ell-1}^{(3)} \,|\, W ] &=& \hu_{\ell-1}(\hX_{\ell-1,\otau},\otau),
\end{eqnarray*}
and therefore, using Assumption 3 together with 
Assumption 1's Lipschitz property for $u(x,t)$ we obtain
\begin{equation}
\left| \EE[\, p_\ell^{(3)} \!-\! p_{\ell-1}^{(3)} \,|\, W ] \right|
\ \leq\  
c_u (h_\ell^{1/2} + h_{\ell-1}^{1/2} )
+ L_u \left( \| \hX_{\ell,\utau}\!-\!\hX_{\ell-1,\utau} \|_2 + h_{\ell-1} \right)
\label{eq:p3}
\end{equation}

Combining the results in Equations (\ref{eq:p1})-(\ref{eq:p3}), 
there exist constants $c_1$, $c_2$, $c_3$, independent 
of $W$, such that
\[
\left| \rule{0in}{0.16in} \EE[\hP_\ell \!-\! \hP_{\ell-1} \,|\, W] \right|
\ \leq \
c_1 \, \| \Delta \hX \|_\infty + 
c_2 \, \| \hX_{\ell,\utau}\!-\!\hX_{\ell-1,\utau} \|_2 + 
c_3 \, h_{\ell-1}^{1/2}.
\]
Hence, using the results from Lemma \ref{lemma:diff}, there 
exists a constant $c$ such that
\[
\VV\left[ \EE[\hP_\ell \!-\! \hP_{\ell-1} \, | \, W ] \right]
\ \leq\ 
\EE\left[ \left( \EE[\hP_\ell \!-\! \hP_{\ell-1} \, | \, W ] \right)^2\right]
\ \leq\ c \, h_{\ell-1} \, |\log h_{\ell-1}|
\]
for sufficiently small $h_{\ell-1}$.  

This proves the first bound in the statement of the theorem.
For the second bound, we note that
\begin{eqnarray*}
\lefteqn{ \VV\left[\, \hP_\ell \!-\! \hP_{\ell-1} \, | \, W \, \right] } \\
& = & \VV\left[ p_\ell^{(2)}(W)\, p_\ell^{(3)}(W,Z) 
              - p_{\ell-1}^{(2)}(W)\, p_{\ell-1}^{(3)}(W,Z) \,|\, W \right] \\
& = & \left(p_\ell^{(2)}(W)\right)^2 \VV\left[p_\ell^{(3)}(W,Z) \,|\, W \right]
+ \left(p_{\ell-1}^{(2)}(W)\right)^2 \VV\left[p_{\ell-1}^{(3)}(W,Z) \,|\, W  \right]
\end{eqnarray*}
The first line in the above is because $p_\ell^{(1)}(W)$ and 
$p_{\ell-1}^{(1)}(W)$ do not vary with $Z$, and so do not contribute
to the variance.
The second line comes from the fact that at most only one of the 
coarse and fine paths is split, and so at most only one of 
$p_\ell^{(3)}(W,Z)$ and $p_{\ell-1}^{(3)}(W,Z)$ has a non-zero variance.

The terms $p_\ell^{(2)}(W)$ and $p_{\ell-1}^{(2)}(W)$ are each bounded 
by $\exp(T \|V\|_\infty)$.  Using Theorem \ref{thm:var} to bound
$\VV[p_\ell^{(3)}(W,Z)]$ and $\VV[p_{\ell-1}^{(3)}(W,Z)]$ gives the result
that there exists a constant $c$ such that
\[
\VV\left[\, \hP_\ell \!-\! \hP_{\ell-1} \,|\, W \, \right] 
\ \leq\ c\ 
\EE[\max(\htau_\ell,\htau_{\ell-1})-\otau \,|\, W ]
\ \leq\ c\  
\EE[ \, |\htau_\ell \!-\! \htau_{\ell-1}| \,|\, W ].
\]
The above inequalities hold for each value of the random variable $W$.
Taking an expectation, and using the result in Lemma 
\ref{lemma:split_exit} then gives the final result that
\[
\EE\left[\rule{0in}{0.16in}
\VV\left[\, \hP_\ell \!-\! \hP_{\ell-1} \, | \, W \, \right] \right]
= O( h_{\ell-1}^{1/2} |\log h_{\ell-1}|^{1/2}).
\]
\end{proof}

Using the bounds from this theorem, we note that 
the two terms in the bound for the variance $V_\ell$
in Equation (\ref{eq:Vl}) are balanced when
$M_\ell = O( h_{\ell-1}^{-1/2} |\log h_{\ell-1}|^{-1/2})$.


\subsection{Expected complexity}

Finally, given the previous results we can obtain the 
following result on the expected complexity to achieve an
$\eps$ root-mean-square error.

\begin{corollary}
Under the given assumptions, a RMS error of $\eps$ can be 
achieved with an $O(\eps^{-2} |\log \eps|^3)$ expected 
computational cost.
\end{corollary}

\begin{proof}
The proof is very similar to the standard multilevel 
complexity analysis in \cite{giles08}, but has to be modified 
slightly because the additional $\log h_\ell$ factors 
in the variance mean it does not quite have the usual 
form expected in condition iii) in Theorem \ref{thm:MLMC}.

On level $\ell\!\geq\! 1$, we choose to use $h_\ell = 4^{-\ell} h_0$ 
with $M_\ell = \lceil 2^\ell / \ell^{1/2} \rceil$ paths 
in the splitting estimator (with $\lceil x \rceil$ denoting 
$x$ rounded up to the nearest integer).
Recalling from Assumption 2 that there is assumed to be 
a unit computational cost for each timestep, the computational
cost of one sample is equal to the sum of the number of coarse 
and fine timesteps before the first of the two paths exits 
the domain, plus the combined sum of the timesteps for the 
subsequent $M_\ell$ split coarse or fine paths.

Hence the expected cost of one sample with splitting
on level $\ell\!\geq\! 1$ is
\[
C_\ell\ \leq\ (h_\ell^{-1}+h_{\ell-1}^{-1}) 
\, \EE[ \min(\htau_\ell,\htau_{\ell-1}) ]
+ M_\ell\, h_\ell^{-1}\,  \EE[ | \htau_\ell - \htau_{\ell-1}| ]
\ =\ O(h_\ell^{-1}),
\]
since $M_\ell \!=\! O(h_\ell^{-1/2} |\log h_\ell|^{-1/2})$ and Lemma 9 gives 
$\EE[ | \htau_\ell - \htau_{\ell-1}| ] \!=\! O(h_\ell^{1/2} |\log h_\ell|^{1/2})$.

Given $O(h_\ell^{1/2})$ weak convergence, to reduce the bias 
to $\eps/\sqrt{2}$ requires $h_L \!=\! O(\eps^2)$ on the finest 
level, while for $O(h_\ell)$ weak convergence using the 
Gobet \& Menozzi correction we require $h_L \!=\! O(\eps)$.
In either case, we have $L \!=\! O(|\log \eps|)$. 

Combining Equation (\ref{eq:Vl}) with the results from 
Theorem \ref{thm:main}, the multilevel variance for 
levels $\ell\!\geq\! 1$  is
\[
V_\ell\ = O(h_{\ell-1} |\log h_{\ell-1}|) = O(h_{\ell-1}\, \ell).
\]
Following the analysis in \cite{giles15},
choosing
\[
N_\ell = 2\, \eps^{-2} \left\lceil \left( \sum_{\ell'=0}^L \sqrt{C_{\ell'} V_{\ell'}} \right)
                          \sqrt{V_\ell / C_\ell} \right\rceil,
\]
to ensure that the overall variance is less than $\fracs{1}{2}\eps^2$,
the expected total cost is bounded by
\[
C_{tot}\ \leq\ 2\, \eps^{-2} \left( \sum_{\ell=0}^L \sqrt{C_\ell V_\ell} \right)^2
+ ~ \sum_{\ell=0}^L C_\ell.
\]
Since $C_\ell V_\ell = O(\ell)$, there exists a constant $c$ such that
\[
\sum_{\ell=0}^L \sqrt{C_\ell V_\ell}
\ \leq\  
c\ \sum_{\ell=0}^L \ell^{1/2}
\ \leq\  
c\ \int_0^{L+1} \ell^{1/2} \, \D \ell
\ =\ \frac{2c}{3} \, (L\!+\!1)^{3/2}.
\]
Because $\displaystyle \sum_{\ell=0}^L C_\ell = O(\eps^{-2})$, 
it follows that the total cost is $O(\eps^{-2} |\log \eps|^{3})$.
\end{proof}

Comment:  the analysis above uses
$M_\ell\!=\! \lceil 2^\ell / \ell^{1/2} \rceil$.
If instead, we use $M_\ell\!=\! 2^\ell$ then
$C_\ell \!=\! O(h_\ell^{-1}  |\log h_\ell|^{1/2})$ 
and the overall complexity becomes 
$O(\eps^{-2} |\log \eps|^{7/2})$.  This is 
poorer, but if the variance analysis 
is not sharp, and in fact 
$\VV[ \EE[\hP_\ell \!-\! \hP_{\ell-1} \, | \, W] \, ] \!=\! O(h_\ell)$
and
$\EE[ \VV[\hP_\ell \!-\! \hP_{\ell-1} \, | \, W] \, ] \!=\! O(h_\ell^{1/2})$,
then the choice $M_\ell\!=\! 2^\ell$
is asymptotically optimal and would give a 
complexity of
$O(\eps^{-2} |\log \eps|^{2})$.
Hence, the numerical experiments use 
$M_\ell\!=\! 2^\ell$.

\section{Numerical experiments}


The testcase is simple Brownian diffusion in a 3D cube, 
$D = [-1, +1]^3$, over a unit time interval, $[0,1]$.
The initial data is $x_0 \!=\! {\bf 0}$, and the output quantity
of interest is the expected exit time.  For the PDE
formulation this corresponds to the PDE
\[
\frac{\partial u}{\partial t} 
+ \fracs{1}{2} \, \nabla^2 u + 1 = 0,
\]
subject to homogeneous boundary data on $\partial D$
and $t\!=\!1$. Using a standard Fourier series expansion,
adjusted to the homogeneous boundary conditions on $\partial D$,
and noting the symmetry in each coordinate direction,
the solution at earlier times $0\!\leq\! t \!<\! 1$ 
can be expressed as
\[
u(x,t) = \sum_{{\rm odd }\, i,j,k \geq 1} A_{i,j,k}(t)
 \cos(i \pi x_1/2) \cos(j \pi x_2/2) \cos(k \pi x_3/2)
\]
where the amplitudes $A_{i,j,k}(t)$ satisfy the 
ordinary differential equation
\[
\frac{\D A_{i,j,k}}{\D t} 
- \frac{(i^2\!+\!j^2\!+\!k^2)\, \pi^2}{8} A_{i,j,k} 
+ \frac{64 \, (-1)^{(i+j+k+1)/2}}{i\,j\, k\, \pi^3} = 0,
\]
subject to terminal condition $A_{i,j,k}(1) \!=\! 0$.
Therefore
\[
A_{i,j,k}(t) = 
\frac{512 \, (-1)^{(i+j+k+1)/2}}{i\,j\, k\, (i^2\!+\!j^2\!+\!k^2)\, \pi^5}
\left(1 - 
\exp\left(- \frac{8\ (1\!-\!t)}{(i^2\!+\!j^2\!+\!k^2)\, \pi^2}\right)
\ \right),
\]
which gives $u({\bf 0},0) \approx 0.435930$.

On the coarsest level the timestep is chosen to be $h_0=0.1$, 
so that the mean exit time corresponds to just 4 timesteps.  On finer
levels it is $h_\ell = 4^{-\ell}\, h_0$.  $M_\ell\!=\! 2^\ell$ sub-samples
are used when approximating the conditional expectations once either 
the coarse or fine path has left the domain.

Numerical results are obtained with 3 schemes:
\begin{itemize}
\item
(orig) -- the original MLMC method proposed by 
Higham {\it et al} \cite{hmrsy13}
\item
(new1) -- the new method with sub-sampling
to approximate the conditional expectations
\item
(new2) -- the new method which in addition uses the 
Gobet \& Menozzi (GM) boundary shift, in which a path on level 
$\ell$ is considered to have crossed the boundary if 
$\| x \|_\infty \!>\! 1 \!-\! c_0\, h^{1/2}_\ell$ where 
$c_0 \!=\! -\zeta(1/2)/\sqrt{2 \pi} \!\approx\! 0.5826$.
\end{itemize}

\begin{figure}[t!]
\begin{center}
\includegraphics[width=0.9\textwidth]{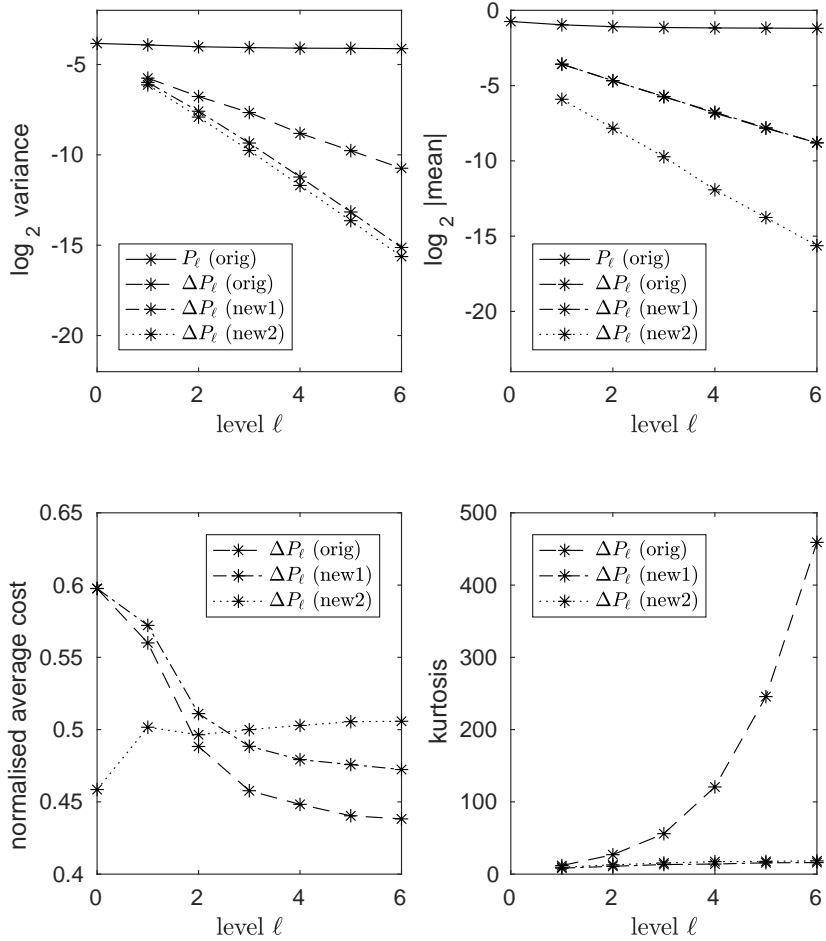}
\end{center}

\caption{Testcase: MLMC variance $V_\ell$, mean correction 
$\EE[\oP_\ell\!-\!\oP_{\ell-1}]$, normalised cost per path sample, 
and kurtosis for three MLMC methods, and standard 
single level Monte Carlo.}
\label{fig:fk1_1}
\end{figure}

Figure \ref{fig:fk1_1} compares results from the three schemes, and 
also a standard single level Monte Carlo simulation without the
Gobet \& Menozzi boundary correction.

The top left plot shows the decay in the MLMC variance $V_\ell$.
This improves from approximately $O(h_\ell^{1/2})$ with the original 
method to approximately $O(h_\ell)$ with sub-sampling. Adding the
GM correction reduces the variance slightly.
The top right plot shows the convergence of $\EE[\hP_\ell \!-\! \hP_{\ell-1}]$.
As expected the introduction of sub-sampling does not affect this,
to within Monte Carlo sampling error, but the GM correction improves
the rate of convergence from $O(h_\ell^{1/2})$ to $O(h_\ell)$.

The bottom left plot shows the average cost per path 
at each level, defined as the number of generated Normal random 
numbers expressed as a fraction of the number $3/h_\ell$ which 
would be required for a single fine path to reach time $t\!=\!1$.
As expected, the cost for the original method is approximately
equal to the expected exit time.  With sub-sampling, the cost is
slightly higher, up to 15\% when using the GM correction, but the 
additional cost is not significant
and in particular does not increase with level $\ell$.

The bottom right plot shows the computed kurtosis of 
$\oP_\ell \!-\! \oP_{\ell-1}$.  The increasing value for the original
method is a clear symptom of the fact that for a few samples the 
exit times of the coarse and fine paths are very different.  
With sub-sampling, there is no asymptotic increase in the kurtosis.

\begin{figure}[t!]
\begin{center}
\includegraphics[width=0.9\textwidth]{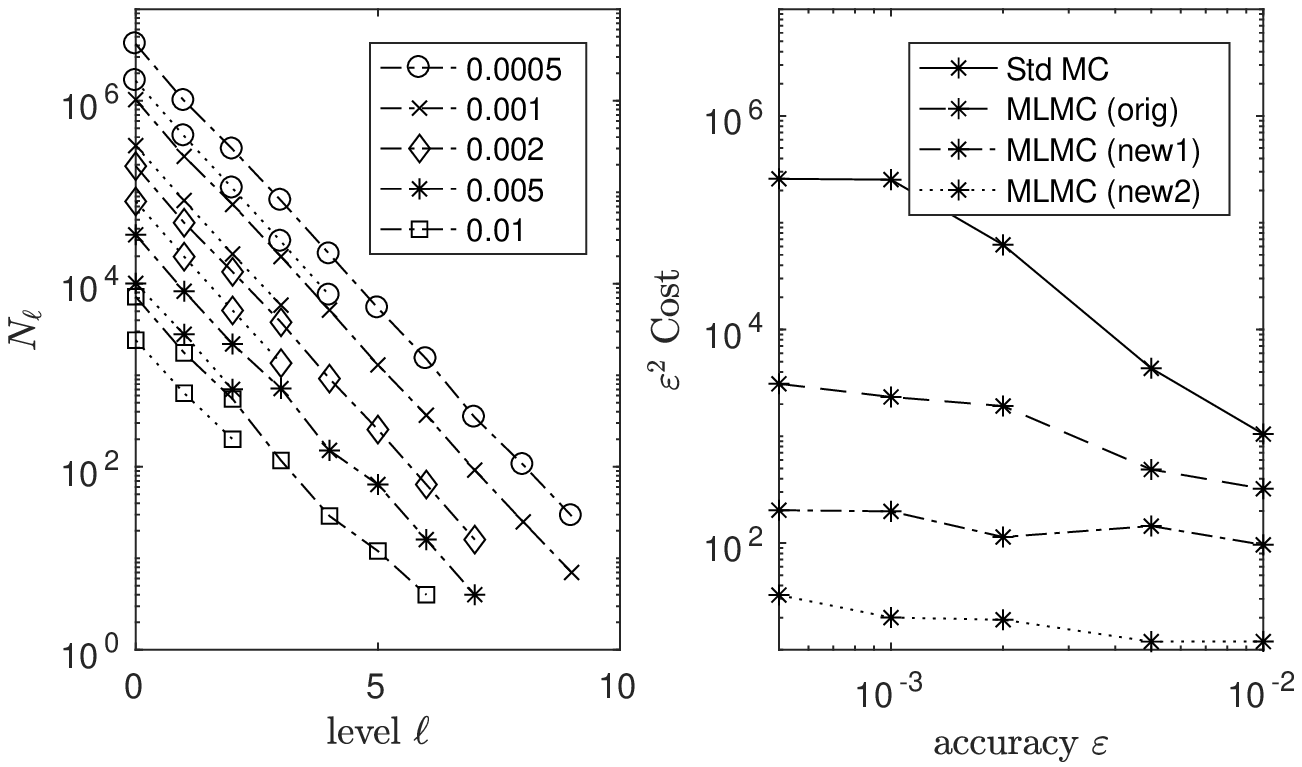}
\end{center}

\caption{Testcase: number of samples on different levels for
different accuracies when using two new MLMC methods 
(dash-dot lines: new1, dotted lines: new2, and cost vs.~accuracy
comparison of three MLMC methods and the standard single level Monte Carlo.}
\label{fig:fk1_2}
\end{figure}

Figure \ref{fig:fk1_2} presents MLMC results for a number 
of different user-specified accuracies $\eps$. The MLMC 
methodology in selecting the near-optimal number of samples
on each level, and finest level $L$, follows the description
given in \cite{giles15}.
The plot on the left shows the number of path samples used 
on level $\ell$ by the best two methods, (new1: dash-dot lines) 
and (new2: dotted lines). As is standard in MLMC, 
for larger values of $\eps$ it is not necessary to use such 
a large value for $L$ to ensure that the weak error (or bias) 
contribution to the desired Mean Square Error is sufficiently 
small.  It is particularly noticeable that the superior weak 
convergence of the Gobet \& Menozzi correction in (new2)
greatly reduces the number of levels needed compared to (new1).

The plot on the right shows the total cost (defined as the 
number of Normal random numbers used) multiplied by 
$\eps^2$.  With the new sub-sampling, this scaled cost
is almost independent of the desired accuracy, which is
consistent with the theory which predicts the cost should 
be $O(\eps^{-2} L^3)$.  Because of the improved weak 
convergence, (new2) uses only half as many levels as 
(new1) and consequently the cost is reduced by a factor 
of 6-8.

\section{Conclusions and extensions}

In this paper we have developed an efficient multilevel
Monte Carlo algorithm for the estimation of mean exit times 
and other associated functionals for stopped diffusions 
within a bounded domain. A root-mean-square
accuracy of $\eps$ is achieved with a computational
complexity which is $O(\eps^{-2} |\log \eps|^3)$.  
The key feature of the algorithm is the use of an 
approximated conditional expectation once either the 
coarse or fine path has exited the domain. This greatly 
reduces the multilevel correction variance without 
significantly increasing the computational cost per path.

There are several directions in which this work can be 
extended.  The first is to linear functionals of the 
PDE solution, such as
\[
P = \int_0^T \int_D \rho(x,t) \ u(x,t) \ \D x\, \D t.
\]
If we take $\rho(x,t)$ to be normalised so that
$\int_0^T \int_D \rho(x,t) \D x\, \D t = 1$, then this is very simply
treated by randomising the starting point for the SDE paths,
taking the starting point $(x_0,t_0)$ from the probability
distribution with density $\rho(x,t)$.  There are similar 
extensions for linear functionals of the solution on the 
boundary.

The next possible extension is to nonlinear functionals of the 
form
\[
P = \int_0^T \int_D \rho(x,t) \ F(u(x,t)) \ \D x\, \D t,
\]
for some nonlinear function $F(u)$.  This case can be tackled 
by representing it as a nested expectation
\[
P = \EE\left[ F(u(x_0,t_0)) \, \right]
\]
with the outer expectation over the starting point $(x_0,t_0)$
and $u(x_0,t_0)$ given by the usual inner expectation from 
(\ref{eq:functional}).  Multilevel techniques for treating 
nested simulations are discussed in \cite{giles15} based on
prior research by Bujok, Hambly \& Reisinger \cite{bhr15}.  
In particular, the multi-index Monte Carlo (MIMC) method 
\cite{hnt16} may be very helpful in this context.

Another direction for extension is to Neumann boundary conditions.
If an application has only homogeneous Neumann boundary conditions,
so the Feynman-Kac functional depends only on the interior path,
then the extension should be straightforward as it is easy to 
simulate the reflected diffusion.  The difficulty will arise if
there are mixed Neumann / Dirichlet boundary conditions.  In this 
case, the MLMC variance could be larger because of the problem 
of one of the fine/coarse path pairs hitting the Dirichlet boundary
while the other hits the Neumann boundary.  In addition, the PDE 
solution may lack the necessary regularity assumed in this paper.

Finally, we could consider the solution of elliptic PDEs.  This 
corresponds to parabolic solutions in the limit in which the 
time duration is extended to infinity.  Multilevel techniques 
for handling this limit are presented and analysed
in \cite{fg17}, as well as being discussed in Section 10.1 in 
\cite{giles15}, based on prior research by Glynn \& Rhee 
on limiting distributions from Markov chains \cite{gr14}.
However, there is a difficulty in the numerical analysis, 
and possibly in practice, which is that the standard strong error 
analysis for a finite time interval $[0,T]$, includes a factor 
which grows exponentially with respect to $T$.  Only under certain 
conditions (see \cite{fg17}), similar to the contracting 
Markov chains considered by Glynn \& Rhee, does this not occur, 
and this may restrict the use of MLMC for this application.

\section*{Acknowledgments}

The authors are very grateful to Abdul-Lateef Haji-Ali, 
Wei Fang and the anonymous referees for their feedback which 
led to significant improvements in the paper.

MBG was funded in part by EPSRC grant EP/H05183X/1, and
FB acknowledges Portuguese FCT funding under grant SFRH/BPD/79986/2011.
In compliance with EPSRC's open access initiative, the data in 
this paper, and the codes which generated it, are available
from doi.org/10.5287/bodleian:XmZwK1zzv.

\bibliographystyle{plain}   
\bibliography{../../bib/mc,../../bib/mlmc}

\end{document}